\documentclass[final,12pt,a4paper]{amsart} 

\usepackage{graphicx}
\usepackage[all]{xy}
\usepackage{placeins}
\usepackage{enumitem}
\usepackage{amssymb}
\usepackage{latexsym}
\usepackage{amsmath}
\usepackage{mathrsfs}
\usepackage{array,booktabs}
\usepackage{verbatim}
\usepackage{fullpage}
\usepackage[notref, notcite]{showkeys}  
\usepackage{color}
\usepackage[normalem]{ulem}

\usepackage{hyperref}
\newcommand{\harxiv}[1]{ \href{http://arxiv.org/abs/#1}{\texttt{arXiv:#1}}}
\newcommand{\hyref}[2]{ \hyperref[#2]{#1~\ref*{#2}} } 

\theoremstyle{plain}
\newtheorem{theorem}{Theorem}

\newtheorem{lemma}[theorem]{Lemma}

\newtheorem{proposition}[theorem]{Proposition}
\newtheorem{introtheorem}{Theorem}
\newtheorem{introcorollary}[introtheorem]{Corollary}

\theoremstyle{definition}
\newtheorem{remark}[theorem]{Remark}

\newtheorem*{definition}{Definition}

\newcommand{\sC}{\mathsf{C}}
\newcommand{\sD}{\mathsf{D}}
\newcommand{\sH}{\mathsf{H}}
\newcommand{\sK}{\mathsf{K}}
\newcommand{\sM}{\mathsf{M}}
\newcommand{\sN}{\mathsf{N}}
\newcommand{\sR}{\mathsf{R}}
\newcommand{\sS}{\mathsf{S}}
\newcommand{\sT}{\mathsf{T}}
\newcommand{\sU}{\mathsf{U}}
\newcommand{\sV}{\mathsf{V}}
\newcommand{\sX}{\mathsf{X}}
\newcommand{\sY}{\mathsf{Y}}

\DeclareMathAlphabet{\mathpzc}{OT1}{pzc}{m}{it}

\newcommand{\bC}{\mathbb{C}}
\newcommand{\bH}{\mathbb{H}}
\newcommand{\bN}{\mathbb{N}}
\newcommand{\bP}{\mathbb{P}}
\newcommand{\bZ}{\mathbb{Z}}

\newcommand{\mcF}{\mathcal{F}}
\newcommand{\mcS}{\mathcal{S}}
\newcommand{\mcT}{\mathcal{T}}

\newcommand{\Db}{\sD^b}
\newcommand{\Kb}{\sK^b}
\newcommand{\proj}{\mathsf{proj}}
\renewcommand{\mod}{\mathsf{mod}}
\newcommand{\Mod}{\mathsf{Mod}}
\newcommand{\Fac}{\mathsf{Fac}}
\newcommand{\Stab}{\mathsf{stab}}

\newcommand{\kk}{\mathbf{k}}

\DeclareMathOperator{\Hom}{\mathrm{Hom}}
\DeclareMathOperator{\End}{\mathrm{End}}

\newcommand{\extn}[1]{\langle #1 \rangle}
\DeclareMathOperator{\add}{\mathsf{add}}
\DeclareMathOperator{\thick}{\mathsf{thick}}
\newcommand{\orth}{{}^\perp}

\newcommand{\iffdef}{\stackrel{\text{def}}{\iff}}
\newcommand{\rightlabel}[1]{\stackrel{#1}{\to}}
\newcommand{\longmapsfrom}{\longleftarrow\!\shortmid}
\renewcommand{\longmapsto}{\shortmid\!\longrightarrow}

\renewcommand{\phi}{\varphi}
\renewcommand{\epsilon}{\varepsilon}

\begin{document}

\title[Silting discrete]{Contractibility of the stability manifold for silting-discrete algebras}

\author{David Pauksztello} 
\address{Dipartimento di Informatica, Universit\`a degli Studi di Verona, Strada le Grazie 15 - Ca' Vignal 2, 37134 Verona, Italy.}
\email{david.pauksztello@univr.it}

\author{Manuel Saor\'in}
\address{Departamento de Matem\'aticas, Universidad de Murcia, Aptdo. 4021, 30100 Espinardo, Murcia, Spain.}
\email{msaorinc@um.es}

\author{Alexandra Zvonareva}
\thanks{Alexandra Zvonareva is supported by the RFBR Grant 16-31-60089. Manuel Saor\'in is supported by research projects from the Ministerio de Econom\'ia y Competitividad of Spain (MTM2016-77445P) and from the Fundaci\'on 'S\'eneca' of Murcia (19880/GERM/15), both with a part of FEDER funds}
\address{Chebyshev Laboratory, St. Petersburg State University, 14th Line 29B, St. Petersburg 199178, Russia.}
\email{alexandrazvonareva@gmail.com}

\keywords{Bounded t-structure, silting-discrete, stability condition}

\subjclass[2010]{18E30, 16G10}

\begin{abstract}
We show that any bounded t-structure in the bounded derived category of a silting-discrete algebra is algebraic, i.e. has a length heart with finitely many simple objects. As a corollary, we obtain that the space of Bridgeland stability conditions for a silting-discrete algebra is contractible.
\end{abstract}

\maketitle

\section*{Introduction}

Stability conditions on triangulated categories were introduced by Bridgeland in \cite{Bridgeland} as a means of extracting geometry from homological algebra with a view to constructing moduli spaces arising in the context of Homological Mirror Symmetry. They can be thought of as a continuous generalisation of bounded t-structures. The main result of \cite{Bridgeland} asserts that the space of stability conditions form a complex manifold, the \emph{stability manifold}. This can be thought of as geometrically encoding most of the cohomology theories on a given triangulated category.

Bounded t-structures admit a mutation theory given by HRS-tilts (see Proposition~\ref{prop:HRS} below), giving rise to a graph that is closely related to the exchange graphs occurring in cluster combinatorics \cite{KQ}, which is the skeleton of the stability manifold in the Dynkin case. Despite being the focus of extensive investigation, for example \cite{Bridgeland,DDC2,DHKK,DK,HKK,Macri,Okada,Qiu,QW}, computations with stability conditions are difficult. For example, it is widely believed that whenever the stability manifold is nonempty it is contractible. However, this has been proved in only few cases, though the list is now growing, see \cite{DDC2,DK,HKK,Macri,Okada,QW}.

Silting objects are a generalisation of tilting objects due to Keller and Vossieck in \cite{KV}. In the context of bounded derived categories of finite-dimensional algebras, silting objects enable the detection of t-structures whose hearts are equivalent to module categories of finite-dimensional algebras \cite{KY}. \emph{Silting-discreteness} \cite{Aihara} is a finiteness condition on a triangulated category that says there are only finitely many silting objects in any interval in the poset of silting objects \cite{AI}; see below for precise definitions. Examples of silting-discrete finite-dimensional algebras include hereditary algebras of finite representation type, derived-discrete algebras \cite{DDC2}, preprojective algebras of Dynkin type \cite{AM}, symmetric algebras of finite representation type \cite{Aihara}, Brauer graph algebras whose Brauer graphs contain at most one cycle of odd length and no cycles of even length \cite{AAC}, and local algebras \cite{AI}.

The purpose of this note is to establish the following characterisation of the bounded t-structures in the bounded derived category $\Db(\Lambda)$ of a silting-discrete finite-dimensional algebra $\Lambda$. We recall that a bounded t-structure is \emph{algebraic} if it is given by a silting object; see Section~\ref{sec:background} for the precise definition.

\begin{introtheorem} \label{thm:main}
If $\Lambda$ is a silting-discrete finite-dimensional $\kk$-algebra, then any bounded t-structure in $\Db(\Lambda)$ is algebraic, i.e. has a length heart.
\end{introtheorem}

This result means that the techniques and methods used in \cite{DDC2,QW} to show that the stability manifold of a derived-discrete algebra is contractible can be applied here.

\begin{introcorollary} \label{cor:contractible}
If $\Lambda$ is a silting-discrete finite-dimensional $\kk$-algebra, then the stability manifold $\Stab(\Db(\Lambda))$ is contractible.
\end{introcorollary}

In forthcoming work \cite{AMY}, T. Adachi, Y. Mizuno and D. Yang independently obtain similar results in the setting of silting-discrete triangulated categories.

The outline of this note is as follows. In Section~\ref{sec:background} we recall the concepts and results that will be necessary to establish Theorem~\ref{thm:main}. In Section~\ref{sec:proof} we prove Theorem~\ref{thm:main}. Once one has Theorem~\ref{thm:main} the proof of Corollary~\ref{cor:contractible} is implicit in \cite{DDC2,QW}. For the convenience of the reader we briefly sketch the narrative of the argument in \cite{DDC2,QW} in Section~\ref{sec:stability}.

\smallskip

\noindent{\bf Convention.} Throughout this note all subcategories will be full and strict, $\kk$ will be a field, and all algebras will be finite-dimensional $\kk$-algebras. Throughout $\sD$ will be a triangulated category and the shift functor will be denoted by $[1] \colon \sD \to \sD$.

\section{Background} \label{sec:background}

For a subcategory $\sS$ of a triangulated category $\sD$ we define
\[
\sS[>n] = \{S[i] \mid S \in \sS, i>n \} \text{ and }
\sS[<n] =  \{S[i] \mid S \in \sS, i<n \},
\]
analogously for $\sS[\leq n]$ and $\sS[\geq n]$.
For subcategories $\sX$ and $\sY$ of $\sD$ we define
\[
\sX * \sY = \{ D \in \sD \mid \text{there exists a triangle } X \to D \to Y \to X[1] \text{ with } X \in \sX \text{ and } Y \in \sY \}.
\]
A subcategory $\sX$ is \emph{extension closed} if $\sX = \sX * \sX$. We shall denote the extension closure of $\sX$ by $\extn{\sX}$.
We define the \emph{right} and \emph{left perpendicular categories} of $\sX$ by
\begin{eqnarray*}
\sX\orth & = & \{ D \in \sD \mid \Hom_{\sD}(X,D) = 0 \text{ for all } X \in \sX \}; \\
\orth\sX & = & \{ D \in \sD \mid \Hom_{\sD}(D,X) = 0 \text{ for all } X \in \sX \}.
\end{eqnarray*}
For subcategories of an abelian category $\sH$ we use the same notation for the analogous definitions, using short exact sequences instead of triangles.

\subsection{Torsion pairs and t-structures}

The general notion of a torsion pair on an abelian category goes back to \cite{Dickson}.

\begin{definition}
A torsion pair in an abelian category $\sH$ consists of a pair of full subcategories $(\mcT,\mcF)$ such that
$\mcT\orth = \mcF$, $\mcT = \orth\mcF$, and $\sH = \mcT * \mcF$.
We call $\mcT$ the \emph{torsion class} and $\mcF$ the \emph{torsionfree class} of the torsion pair.
\end{definition}

If the abelian category $\sH$ is $\mod(\Lambda)$ for a finite-dimensional algebra $\Lambda$, then any subcategory $\mcT$ closed under extensions, factor objects and direct summands gives rise to a torsion class of a torsion pair; see, e.g. \cite[Ch. VI]{ASS}. A dual statement holds for \emph{torsionfree classes}. For $M \in \sH$ we write $\Fac(M)$ for the smallest torsion class containing $M$. 

The analogue of a torsion pair in a triangulated category is a t-structure \cite{BBD}.

\begin{definition}
A \emph{t-structure} on a triangulated category $\sD$ consists of a pair of full subcategories $(\sX,\sY)$ such that 
$\sX\orth = \sY$, $\sX = \orth\sY$, $\sD = \sX * \sY$ and $\sX[1] \subseteq \sX$ (equivalently, $\sY[-1] \subseteq \sY$).
The subcategory $\sH = \sX \cap \sY[1]$ is an abelian subcategory of $\sD$ called the \emph{heart} of $(\sX,\sY)$.
A t-structure is called \emph{bounded} if
\[
\sD = \bigcup_{n \in \bZ} \sX[n] = \bigcup_{n \in \bZ} \sY[n].
\]
For a bounded t-structure $(\sX,\sY)$ we have $\sX = \extn{\sH[\geq 0]}$ and $\sY = \extn{\sH[<0]}$.
A t-structure is called \emph{algebraic} if it is bounded and $\sH$ is a length category, i.e. $\sH$ has finitely many isomorphism classes of simple objects and each object of $\sH$ is both artinian and noetherian.
\end{definition}

There is a close connection between torsion pairs and t-structures. 

\begin{proposition}[{\cite{BR,Pol,Woolf}}] \label{prop:HRS}
Suppose $(\sX,\sY)$ is a t-structure on $\sD$ with heart $\sH$. 
Then there is a bijection
\begin{eqnarray*}
\{\text{t-structures } (\sX',\sY') \text{ with } \sX[1] \subseteq \sX' \subseteq \sX \} & \stackrel{1-1}{\longleftrightarrow} & \{ \text{torsion pairs } (\mcT,\mcF) \text{ in } \sH \}; \\
(\sX', \sY')                                                                                                              & \longmapsto                                 & (\mcT = \sH \cap \sX', \mcF = \sH \cap \sY'); \\
(\sX' = \extn{\mcT, \sX[1]}, \sY' = \extn{\sY,\mcF})                                                 & \longmapsfrom                             & (\mcT, \mcF).
\end{eqnarray*}
\end{proposition}

The t-structure $(\sX',\sY')$ in Proposition~\ref{prop:HRS} is called an \emph{HRS-tilt} of $(\sX,\sY)$ at the torsion pair $(\mcT,\mcF)$ and is called \emph{intermediate with respect to $(\sX,\sY)$}; see \cite[Prop. I.2.1]{HRS}. Note that $\sX'=\sX[1] * \mcT$ and $Y'=\mcF * \sY$. 

\subsection{Silting, t-structures and $\tau$-tilting}

Silting was first introduced in \cite{KV}; however, we follow the treatment of \cite{AI}.

\begin{definition}
A subcategory $\sS$ of $\sD$ is \emph{silting} if $\thick(\sS) = \sD$ and $\Hom_{\sD}(S,S'[i]) = 0$ for each $S,S' \in \sS$ and $i >0$, where $\thick(\sS)$ is the smallest triangulated subcategory of $\sD$ containing $\sS$ that is closed under direct summands.
An object $S$ of $\sD$ is a \emph{silting object} if $\add(S)$ is a silting subcategory.
\end{definition}

For a finite-dimensional algebra $\Lambda$ we shall freely abuse notation and identify silting subcategories with silting objects, since any silting subcategory in $\Kb(\proj(\Lambda))$ is of the form $\add(S)$, for some silting object uniquely determined up to additive closure.

There is a partial order on silting subcategories \cite{AI}: for silting subcategories $\sS$ and $\sT$,
\[
\sS \geq \sT \iffdef \Hom_{\sD}(S,T[i]) = 0 \text{ for all } S \in \sS, T \in \sT \text{ and } i >0 \iff \sT \subseteq (\sS[<0])\orth .
\]
A silting subcategory $\sT$ is called \emph{two term with respect to $\sS$} if $\sS \geq \sT \geq \sS[1]$, which happens if and only if $\sT \in \sS * \sS[1]$; see, for example, \cite{IJY}.

\begin{definition}[{\cite[Def. 3.6 \& Prop. 3.8]{Aihara}}]
A finite-dimensional algebra $\Lambda$ is \emph{silting-discrete} if for any silting object $S$ and any natural number $n$ there are only finitely many silting objects $T$ such that $S \geq T \geq S[n]$. Note that, via \cite[Lem. 2.14]{QW}, this is equivalent to there being only finitely many silting objects $T$ such that $S \geq T \geq S[1]$.
\end{definition}

In the case that $\sD = \Db(\Lambda)$ for a finite-dimensional algebra $\Lambda$, there is a correspondence between silting subcategories and algebraic t-structures.

\begin{theorem}[{\cite{KY} \& \cite{IJY}}] \label{thm:koenig-yang}
Let $\Lambda$ be a finite-dimensional $\kk$-algebra. Then there is a bijection
\begin{eqnarray*}
\{ \text{silting subcategories of } \Kb(\proj(\Lambda)) \} & \stackrel{1-1}{\longleftrightarrow} & \{ \text{algebraic t-structures on } \Db(\Lambda) \}; \\
\sS                                                                                 & \longmapsto                                  & \big( \sX_\sS = (\sS[<0])\orth , \sY_\sS = (\sS[\geq 0])\orth \big). 
\end{eqnarray*}
Moreover, this restricts to a bijection with intermediate algebraic t-structures,
\[
\{ \text{silting subcategories } \sT \subseteq \sS * \sS[1] \}  \stackrel{1-1}{\longleftrightarrow}  \{ \text{algebraic t-structures } (\sX,\sY) \text{ with } \sX_\sS [1] \subseteq \sX \subseteq \sX_\sS \}.  
\]
\end{theorem}

\begin{definition}
Let $\Lambda$ be a finite-dimensional $\kk$-algebra.
Write $|M|$ for the number of nonisomorphic indecomposable summands of a $\Lambda$-module $M$.
\begin{enumerate}
\item (\cite[Def. 0.1 \& 0.3]{AIR}) A pair $(M,P) \in \mod(\Lambda) \times \proj(\Lambda)$ is a \emph{$\tau$-rigid pair} if $\Hom_\Lambda(M,\tau M) = 0$ and $\Hom_\Lambda(P,M) = 0$. A $\tau$-rigid pair is a \emph{support $\tau$-tilting pair} if $|M| + |P| = |\Lambda|$. If in a support $\tau$-tilting pair $P = 0$, we call $M$ a \emph{$\tau$-tilting module}.
\item (\cite[Def. 1.1]{DIJ}) The algebra $\Lambda$ is \emph{$\tau$-tilting finite} if there are only finitely many isomorphism classes of basic $\tau$-tilting $\Lambda$-modules.
\end{enumerate}
\end{definition}

The following characterisation of support $\tau$-tilting pairs will be useful.

\begin{lemma}[{\cite[Cor. 2.13]{AIR}, see also \cite[Thm. 2.5(3)]{AMV}}] \label{lem:surject}
Let $M \in \mod{\Lambda}$ and $P_1 \rightlabel{\sigma} P_0 \to M \to 0$ be its minimal projective presentation. The pair $(M,P)$ is support $\tau$-tilting if and only if $\Fac(M)$ consists of the $N \in \mod(\Lambda)$ such that $\Hom_\Lambda(\widetilde{\sigma},N)$ is surjective, where $\widetilde{\sigma} = [ \sigma \, 0]$ in the projective presentation $P_1 \oplus P \rightlabel{\widetilde{\sigma}} P_0 \to M$.
\end{lemma}

A result of \cite{DIJ} relates $\tau$-tilting finiteness with functorial finiteness of torsion classes; we refer the reader to, for example \cite{ASS}, for the definition of functorial finiteness.

\begin{theorem}[{\cite[Thm. 3.8]{DIJ}}] \label{thm:functorially-finite}
A finite-dimensional algebra $\Lambda$ is $\tau$-tilting finite if and only if every torsion class (equivalently, every torsionfree class) in $\mod(\Lambda)$ is functorially finite.
\end{theorem}

The results of \cite{AIR} combined with \cite{IJY}  give the following.

\begin{theorem}[{\cite[Thm. 4.6]{IJY} and \cite{AIR}}] \label{thm:tau-tilting}
Let $\sD$ be a Krull-Schmidt, Hom-finite, $\kk$-linear triangulated category and $\sS = \add(S)$ for a silting object $S$. Let $\Gamma = \End_\sD(S)$. Then there is a bijection between the following sets:
\begin{enumerate}
\item basic silting objects $T$ of $\sD$ with $T \in S * S[1]$, modulo isomorphism; and,
\item basic support $\tau$-tilting modules of $\mod(\Gamma)$, and,
\item torsion pairs $(\mcT,\mcF)$ in $\mod(\Gamma)$ in which $\mcT$ and $\mcF$ are functorially finite.
\end{enumerate}
\end{theorem}

\begin{remark} \label{rem:tau-tilting}
Suppose we are in the setup of Theorem~\ref{thm:tau-tilting}. Recall from \cite[Rem. 4.1(ii)]{IJY} there is an equivalence $\Mod(\sS) \simeq \Mod(\Gamma)$, where $\Mod(\sS)$ is the category of contravariant functors from $\sS$ to the category of abelian groups.
Let $(M,P)$ be a support $\tau$-tilting pair of $\mod(\Gamma)$ with minimal projective presentation $P_1 \rightlabel{\sigma} P_0 \to M \to 0$ and the `extended' presentation $\widetilde{P} = P_1 \oplus P \rightlabel{\widetilde{\sigma}} P_0 \to M \to 0$ of Lemma~\ref{lem:surject}. One can uniquely lift this presentation to $\Mod(\sS)$ as $\Hom_\sD(-,\widetilde{S})|_\sS \rightlabel{(-,f)} \Hom_\sD(-,S_0)|_\sS$, where $\Hom_\sD(-,D)|_\sS$ denotes the image of $D$ under the restricted Yoneda functor \cite{Auslander}; cf. \cite[Rem 3.1]{IJY}. The corresponding silting object $T \in S * S[1]$ is then the mapping cone of $f \colon \widetilde{S} \to S_0$ in $\sD$.
\end{remark}

\section{Proof of Theorem~\ref{thm:main}} \label{sec:proof}

We start by showing that when $\Lambda$ is silting-discrete any HRS-tilt of an algebraic t-structure is again algebraic.

\begin{proposition} \label{prop:algebraic}
Let $\Lambda$ be a silting-discrete finite-dimensional algebra. Let $\sS \subseteq \Kb(\proj(\Lambda))$ be a silting subcategory and $(\sX_\sS,\sY_\sS)$ be the corresponding algebraic t-structure on $\Db(\Lambda)$. If $(\sX,\sY)$ is a t-structure intermediate with respect to $(\sX_\sS,\sY_\sS)$ then $(\sX,\sY)$ is algebraic.
\end{proposition}

\begin{proof}
Suppose $(\sX,\sY)$ is a t-structure intermediate with respect to the algebraic t-structure $(\sX_\sS,\sY_\sS)$, where $\sS = \add(S)$ for some basic silting object $S$. 
First observe that since $(\sX,\sY)$ is intermediate with respect to a bounded t-structure $(\sX_\sS,\sY_\sS)$ it is automatically bounded.
Let $\Gamma = \End_{\Kb(\proj(\Lambda))}(S)$ and note that $\sH_\sS \simeq \mod(\Gamma)$ by \cite{KY}. Since $\Lambda$ is silting-discrete, there are finitely many silting objects in $\sS * \sS[1]$, and therefore, by Theorem~\ref{thm:tau-tilting}, finitely many support $\tau$-tilting modules in $\mod(\Gamma)$, whence $\Gamma$ is $\tau$-tilting finite.

By Proposition~\ref{prop:HRS}, there exists a torsion pair $(\mcT,\mcF)$ on $\sH_\sS$ such that $\sX = \extn{\mcT, \sX_\sS [1]}$ and $\sY = \extn{\sY_\sS,\mcF}$. By Theorem~\ref{thm:functorially-finite}, $\mcT$ and $\mcF$ are functorially finite, so that by the correspondence in Theorem~\ref{thm:tau-tilting}, $\mcT = \Fac(M)$ for some support $\tau$-tilting pair $(M,P)$ of $\mod(\Gamma)$, which in turn corresponds to some silting object $T \in \sS * \sS[1]$. By Theorem~\ref{thm:koenig-yang}, this corresponds to an algebraic t-structure $(\sX_\sT,\sY_\sT)$ that is intermediate with respect to $(\sX_\sS,\sY_\sS)$. Invoking Proposition~\ref{prop:HRS} again, there is a torsion pair $(\mcT_\sT,\mcF_\sT)$ on $\sH_\sS$ such that $\sX_\sT = \extn{\mcT_\sT, \sX_\sS [1]}$ and $\sY_\sT = \extn{\sY_\sS,\mcF}$. Furthermore, $\mcT_\sT = \sX_\sT \cap \sH_\sS$.

We claim that $\mcT_\sT = \mcT$. First observe that any $N \in \sH_\sS$ satisfies $N \in (\sT[<-1])\orth$ because $\sT \subseteq \sS * \sS[1]$. Therefore $N \in \sH_\sS$ lies in $\mcT_\sT$ if and only if $N \in (\sT[-1])\orth$.
By Lemma~\ref{lem:surject}, $N \in \Fac(M)$ if and only if $\Hom_{\sH_\sS}(\widetilde{\sigma},N)$ is surjective, where we use the notation of Remark~\ref{rem:tau-tilting}. 
By Remark~\ref{rem:tau-tilting}, we can lift $\widetilde{\sigma}$ to the functor category as $\Hom_\sD(-,\widetilde{S})|_\sS \rightlabel{(-,f)} \Hom_\sD(-,S_0)|_\sS$, and note that via the restricted Yoneda functor (e.g. \cite[Rem. 3.1]{IJY}), $\Hom_{\sH_\sS}(\widetilde{\sigma},N)$ is surjective if and only if 
\[
\scalebox{0.95}{\xymatrix @C=36pt{
\Hom_{\Mod(\sS)}(\Hom_\sD(-,S_0)|_\sS, \Hom_\sD(-,N)|_\sS) \ar[r]^-{(f^*,(-,N))}  & \Hom_{\Mod(\sS)}(\Hom_\sD(-,\widetilde{S})|_\sS,\Hom_\sD(-,N)|_\sS) \\
\Hom_\sD(S_0,N)|_\sS \ar[r]_-{(f,N)} \ar[u]_{\sim}                                                                                                  & \Hom_\sD(\widetilde{S},N)|_\sS \ar[u]_{\sim} 
}}
\]                                                                       
is surjective, where the vertical arrows are given by the Yoneda embedding. 
But since an additive generator of $\sT$ is given as the mapping cone $\widetilde{S} \rightlabel{f} S_0 \to T \to \widetilde{S}[1]$ and $\Hom_\sD(S_0[-1],N)|_\sS = 0$ since $N \in \sH_\sS$, we have $N \in (\sT[-1])\orth$ if and only if $N \in \Fac(M)$. Hence $\mcT_\sT = \mcT$. It follows that $(\sX,\sY) = (\sX_\sT,\sY_\sT)$, i.e. any t-structure intermediate with respect to $(\sX_\sS,\sY_\sS)$ is algebraic.
\end{proof}

We shall need the following straightforward observation; cf. \cite[Lem. 2.9]{QW}.

\begin{lemma} \label{lem:sandwich}
Suppose $(\sX,\sY)$ is a bounded t-structure on $\Db(\Lambda)$ and $(\sX_\sS,\sY_\sS)$ is an algebraic t-structure on $\Db(\Lambda)$. 
There exist integers $m \geq n$ such that $\sX_\sS[m] \subseteq \sX \subseteq \sX_\sS[n]$.
\end{lemma}

\begin{proof}
Note that $\sX_\sS[m] \subseteq \sX \subseteq \sX_\sS[n]$ is equivalent to $\sY_\sS[m] \supseteq \sY \supseteq \sY_\sS[n]$.
Since $(\sX_\sS,\sY_\sS)$ is algebraic, there exists finitely many simple objects $X_1, \ldots, X_t \in \sH_\sS$ such that $\extn{ X_1, \ldots, X_t} = \sH_\sS$.
The boundedness of $(\sX,\sY)$ asserts the existence of an integer $k$ such that $X_i \in \sX[k]$ for each $1 \leq i \leq t$, whence $\sH_\sS [\geq 0] \subseteq \sX[k]$. Thus, $(\sH_\sS [\geq 0])\orth = \sY_\sS \supseteq \sY[k]$, and we can take $m = -k$.
Analogously, there also exists an $l$ such that  $X_i \in \sY[l]$ for each $1 \leq i \leq t$, so that $\sH_\sS [< 0] \in \sY[l-1]$ and $\orth (\sH_\sS [< 0]) = \sX_\sS \supseteq \sX[l-1]$, and we can take $n= 1- l $.
\end{proof}

\begin{lemma} \label{lem:intermediate}
Let $\Lambda$ be a silting-discrete finite-dimensional algebra. Suppose $(\sX,\sY)$ is a bounded t-structure on $\Db(\Lambda)$. Then there exists a silting subcategory $\sS = \add(S)$ and an algebraic t-structure $(\sX_\sS,\sY_\sS)$ such that $\sX_\sS[1] \subseteq \sX \subseteq \sX_\sS$.
\end{lemma}

We give two proofs of this lemma. The first one is tailored to the level of generality of this note, while the second one uses a technique that could possibly be adapted to a more general setting.

\begin{proof}
By Lemma~\ref{lem:sandwich}, there are integers $m \geq n$ and an algebraic t-structure $(\sX_\sT,\sY_\sT)$ such that $\sX_\sT[m] \subseteq \sX \subseteq \sX_\sT[n]$. Choose $m$ minimal and $n$ maximal; without loss of generality we may assume $n=0$.
Following \cite[\S 2]{FMT}, we set $\sY_1 = \sY_\sT[1] \cap \sY$, and $\sX_1=\orth\sY_1$ and claim that $(\sX_1, \sY_1)$ is a t-structure.  
One can check immediately in this case that ${\sY_\sT}[1] \supseteq {\sY_1} \supseteq {\sY_\sT}$ (equivalently, ${\sX_\sT}[1] \subseteq {\sX_1} \subseteq {\sX_\sT}$); and ${\sY_1}[m-1] \supseteq \sY \supseteq {\sY_1}$, i.e. ${\sX_1}[m-1] \subseteq \sX \subseteq {\sX_1}$, and the lemma would follow by induction. It is therefore sufficient to establish the claim.

Let $\mcF = \sH_\sT \cap \sY_1 = \sH_\sT \cap \sY$ and observe that $\mcF$ is closed under extensions and direct summands because $\sH_\sT$ and $\sY_1$ are. Let $F \in \mcF$ be an object and consider a short exact sequence $0 \to F' \to F \to F'' \to 0$ in $\sH_\sT$, which gives rise to a triangle $F''[-1] \to F' \to F \to F''$ in $\Db(\Lambda)$. Since $F'' \in \sH_\sT \subseteq \sY_\sT[1]$, whence $F''[-1] \in \sY_\sT$. Using $\sY_\sT \subseteq \sY_\sT[1]$ and $\sY_\sT \subseteq \sY$, we get $F''[-1] \in \sY_1$. This gives $F' \in \sY_1$, i.e. $\mcF$ is closed under subobjects. Since $\sH_\sT \simeq \mod(\Gamma)$, where $\Gamma = \End(T)$, it follows that $\mcF$ is a torsionfree class giving rise to a torsion pair $(\mcT,\mcF)$ with $\mcT = \orth\mcF$.

By Proposition~\ref{prop:HRS} there is a t-structure $(\widehat{\sX_1},\widehat{\sY_1})$ with $\widehat{\sX_1} = \extn{{\mcT}, {\sX_\sT}[1]}$ and $\widehat{\sY_1} = \extn{{\sY_\sT}, {\mcF}}$. The inclusion $\extn{\sY_\sT, \mcF} \subseteq \sY_1$ is clear. For the other inclusion, take $Y \in \sY_\sT[1] \cap \sY$ and consider the truncation triangle for $Y$ with respect to $(\sX_\sT, \sY_\sT)$:
\[
\tau_{\sX_\sT}Y \to Y \to \tau_{\sY_\sT} Y \to \tau_{\sX_\sT}Y[1].
\]
Since $Y \in \sY_\sT [1]$, we have $\tau_{\sX_\sT} Y \in \sH_\sT$. Since $Y \in \sY$ and $\tau_{\sY_\sT} Y [-1] \in \sY_\sT [-1] \subseteq \sY_\sT \subseteq \sY$ we have $\tau_{\sX_\sT} Y \in \sY$. So $\tau_{\sX_\sT} Y \in \mcF$ and $Y \in \extn{\sY_\sT, \mcF}$, whence $(\sX_1, \sY_1) = (\widehat{\sX_1},\widehat{\sY_1})$ is a t-structure as claimed.
\end{proof}

\begin{proof}[Second proof of Lemma~\ref{lem:intermediate}]
We first need the following lifting and restriction lemma.

\begin{lemma} \label{lem:lift}
If $(\sX,\sY)$ is a t-structure on $\Db(\Lambda)$, then $(\widetilde{\sX},\widetilde{\sY}):=(\orth(\sX\orth),\sX\orth)$ is a t-structure on $\sD(\Lambda)$ such that $(\widetilde{\sX} \cap \Db(\Lambda),\widetilde{\sY} \cap \Db(\Lambda))=(\sX,\sY)$. 
Moreover, if there are integers $m \geq n$ such that $\sX_\sS[m] \subseteq \sX \subseteq \sX_\sS[n]$ for some algebraic t-structure $(\sX_\sS,\sY_\sS)$, then we also have  $\widetilde{\sX_\sS}[m] \subseteq \widetilde{\sX} \subseteq \widetilde{\sX_\sS}[n]$.
\end{lemma}

\begin{proof}
Since $\Db(\Lambda)$ is essentially small, \cite[Cor. 3.5]{SS} says that $(\widetilde{\sX},\widetilde{\sY})$ is indeed a t-structure. Since $\Hom_{\sD(\Lambda)}(\sX,\widetilde{\sY})=0$ by definition, the inclusion $\sX \subseteq \widetilde{\sX}$ holds. Since $\Hom_{\sD(\Lambda)}(\sX,\sY)=0$ we get $\sY \subseteq \widetilde{\sY}$ also. 
Thus, $(\widetilde{\sX} \cap \Db(\Lambda),\widetilde{\sY} \cap \Db(\Lambda))=(\sX,\sY)$. 

Let $\sU = (\sS [<0])\orth$ and $\sV = (\sS [\geq 0])\orth$, where the orthogonals are taken in $\sD(\Lambda)$, making $(\sU,\sV)$ into a silting t-structure in $\sD(\Lambda)$; see \cite{AMV}. We claim that $(\sU, \sV) = (\widetilde{\sX_\sS}, \widetilde{\sY_\sS})$.
Since $\sS$ is a silting subcategory we have $\sS[ \geq 0] \subseteq \sX_\sS$. Thus, $(\sS[ \geq 0])\orth \supseteq \sX_\sS \orth$, i.e. $\widetilde{\sY_\sS} \subseteq \sV$.
For the reverse inclusion, observe that $\sX_\sS \subseteq \sU$, so that $\widetilde{\sY_\sS} = \sX_\sS \orth \supseteq \sU \orth = \sV$.

For the final statement, note that $\sX_\sS [m] \subseteq \sX \subseteq \sX_\sS[n]$ if and only if $(\sX_\sS [m])\orth \supseteq \sX\orth \supseteq (\sX_\sS [n])\orth$, that is $\widetilde{\sY_\sS}[m] \supseteq \widetilde{\sY} \supseteq \widetilde{\sY_\sS}[n]$.
\end{proof}

As in the first proof, we may assume $\sX_\sT[m] \subseteq \sX \subseteq \sX_\sT$ for some algebraic t-structure $(\sX_\sT,\sY_\sT)$. By Lemma~\ref{lem:lift}, we can lift the t-structures and inclusions to $\sD(\Lambda)$; these t-structures restrict to the given t-structures on $\Db(\Lambda)$ and are decorated with tildes.

Again following\cite[\S 2]{FMT}, we set $\widetilde{\sY_1} = \widetilde{\sY_\sT}[1] \cap \widetilde{\sY}$, which, by \cite{BPP,SS}, gives rise to a t-structure $(\widetilde{\sX_1},\widetilde{\sY_1})$.
It has the following properties: $\widetilde{\sY_\sT}[1] \supseteq \widetilde{\sY_1} \supseteq \widetilde{\sY_\sT}$ (equivalently, $\widetilde{\sX_\sT}[1] \subseteq \widetilde{\sX_1} \subseteq \widetilde{\sX_\sT}$); by \cite[Lem. 2.12]{FMT} we have $\widetilde{\sY_1}[m-1] \supseteq \widetilde{\sY} \supseteq \widetilde{\sY_1}$, i.e. $\widetilde{\sX_1}[m-1] \subseteq \widetilde{\sX} \subseteq \widetilde{\sX_1}$.

Now, by Proposition~\ref{prop:HRS}, there exists a torsion pair $(\widetilde{\mcT},\widetilde{\mcF})$ on $\widetilde{\sH_\sT}$ such that $\widetilde{\sX_1} = \extn{\widetilde{\mcT}, \widetilde{\sX_\sT}[1]}$ and $\widetilde{\sY_1} = \extn{\widetilde{\sY_\sS}, \widetilde{\mcF}}$. By \cite[Cor. 3]{NSZ} or \cite[Cor. 4.7]{PV}, $\sH_\sT \simeq \Mod(\Gamma)$, where $\Gamma = \End(T)$. Since any torsion pair on $\Mod(\Gamma)$ restricts to a torsion pair on $\mod(\Gamma)$, the t-structure $(\widetilde{\sX_1},\widetilde{\sY_1})$ restricts to a t-structure $(\sX_1, \sY_1)$ on $\Db(\Lambda)$ such that $\sX_\sT[1] \subseteq \sX_1 \subseteq \sX_\sT$. By Proposition~\ref{prop:algebraic}, $(\sX_1,\sY_1)$ is an algebraic t-structure with  $\sX_1[m-1] \subseteq \sX \subseteq \sX_1$. The lemma now follows by induction.
\end{proof}

\begin{proof}[Proof of Theorem~\ref{thm:main}]
Let $(\sX,\sY)$ be a bounded t-structure in $\Db(\Lambda)$ for a silting-discrete finite-dimensional algebra $\Lambda$. By Lemma~\ref{lem:intermediate}, $(\sX,\sY)$ is intermediate with respect to an algebraic t-structure $(\sX_\sS,\sY_\sS)$. By Proposition~\ref{prop:algebraic}, $(\sX,\sY)$ is therefore algebraic.
\end{proof}

\section{Stability conditions} \label{sec:stability}

Rather than give a formal definition of stability conditions, we give an equivalent formulation due to \cite{Bridgeland}. 
Let $\bH = \{  r \exp(i \pi \phi) \mid r > 0 \text{ and } 0 < \phi \leq 1 \}$. A \emph{stability function} on an abelian category $\sH$ consists of a group homomorphism $Z \colon K_0(\sH) \to \bC$ such that $Z(H) \in \bH$ for each $H \in \sH$. A nonzero object $H \in \sH$ is \emph{semistable} with respect to $Z$ if for each $0 \neq H' \subseteq H$ we have $\phi(Z(H')) \leq \phi(Z(H))$. If $\sH$ is a length category then a stability function is uniquely determined by its action on the simple objects.

\begin{proposition}[{\cite[Prop. 5.3]{Bridgeland}}]
Specifying a stability condition $\sigma = (Z,\sH)$ on a triangulated category $\sD$ is equivalent to specifying a bounded t-structure on $\sD$ together with a stability function $Z \colon K_0(\sH) \to \bC$ on its heart $\sH$ that satisfies the Harder-Narasimhan (HN) property.
\end{proposition}

Since any stability function on a length heart satisfies the HN property, and by Theorem~\ref{thm:main}, all the bounded t-structures in $\Db(\Lambda)$ are algebraic when $\Lambda$ is silting-discrete, we refrain from defining the HN property and refer the reader to \cite{Bridgeland}. From now on, since a bounded t-structure is determined by its heart we shall identify it with its heart.

Each t-structure $\sH$ identifies a `chamber' $\sC_\sH$ of the stability manifold consisting of all stability conditions having that t-structure. If $\sH$ is algebraic, then $\sC_\sH \cong \bH^t$, where $t$ is the number of nonisomorphic simple objects of $\sH$; see \cite{Woolf}. The closure of $\overline{\sC_\sH} = \overline{\bH^t}$.

Recall from \cite{DDC2} that a silting pair $(\sM,\sM')$ consists of a silting subcategory $\sM$ of a triangulated category $\sD$ and a functorially finite subcategory $\sM' \subseteq \sM$. The poset of silting pairs $\bP_2(\sD)$ was defined via the opposite of the following partial order:
\[
(\sN,\sN') \geq (\sM,\sM') \iffdef \sR_{\sM'}(\sM) \geq \sR_{\sN'}(\sN) \geq \sN \geq \sM,
\]
where on the right-hand side the partial order is that from \cite{AI} and $\sR_{\sM'}(\sM)$ is the right mutation of $\sM$ at $\sM'$; see \cite{AI} and \cite[\S 5]{DDC2} for details. One gets the following theorem by observing that the proof in \cite{DDC2} works in this level of generality.

\begin{theorem}[{\cite[Cor. 6.2 \& Thm. 7.1]{DDC2}}] \label{thm:contractible}
Suppose $\Lambda$ is a silting-discrete finite-dimensional algebra. Then $\bP_2(\Kb(\proj(\Lambda)))$ is an CW poset and $B\bP_2(\Kb(\proj(\Lambda)))$, the classifying space of the poset, is contractible.
\end{theorem}

We recall the following from \cite[\S 2]{CW-strat}; see also \cite[\S 2.7]{QW}. Let $X$ be a topological space. Let $e \subseteq X$ be a subspace and denote its closure by $\overline{e}$. A \emph{$k$-cell structure} on $e \subseteq X$ comprises a continuous map $\alpha \colon D \to X$ where $(D^k)^\circ \subseteq D \subseteq D^k$, where $D^k$ is the $k$-disc and $(D^k)^\circ$ is its interior, satisfying $\alpha(D) = \overline{e}$, $\alpha$ restricted to $(D^k)^\circ$ is a homeomorphism onto $e$, and $\alpha$ does not extend to a continuous map with these properties for any larger subspace of $D^k$. In this case $e$ is called a \emph{$k$-cell}. A \emph{cellular stratification} of $X$ is a filtration
\[
\emptyset = X_{-1} \subseteq X_0 \subseteq X_1 \subseteq \cdots \subseteq X_k \subseteq \cdots
\]
such that $X = \bigcup_{k \in \bN} X_k$ and for each $k \in \bN$, $X_k \setminus X_{k-1} = \bigsqcup_{i \in I_k} e_i$ is a disjoint union of $k$-cells.
The \emph{face poset} of \emph{poset of strata}, $P(X)$, of $X$ is defined via the following partial order on its cells:
$e_i \leq e_j$ if and only if  $e_i \subseteq \overline{e_j}$.

Following \cite{QW}, let $\sH$ be (the heart of) an algebraic t-structure and write $\mcS(\sH)$ for the set of isomorphism classes of simple objects of $\sH$. For $I \subseteq \mcS(\sH)$ define
\[
\sC_{\sH,I} = \{ \sigma = (Z,\sH) \mid \phi(Z(S)) = 1 \text{ for } S \in \mcS(\sH) \iff S \in I \}.
\]
This defines a cellular stratification, $\Stab(\Db(\Lambda)) = \bigcup_{\sH} \bigcup_{I \subseteq \mcS(\sH)} \sC_{\sH,I}$, in the case that $\Lambda$ is silting-discrete by Theorem~\ref{thm:main}. The following lemma captures the poset of strata $P(\Stab(\Db(\Lambda))$ algebraically.

\begin{lemma}[{\cite[Cor. 3.10 \& Lem. 3.11]{QW}}] \label{lem:poset}
Let $(\sM,\sM')$ and $(\sN,\sN')$ be silting pairs with corresponding simple objects $\mcS(\sH_\sM)$ and $\mcS(\sH_\sN)$ with $I \subseteq \mcS(\sH_\sM)$ corresponding to $\sM \setminus \sM'$ and  $J \subseteq \mcS(\sH_\sN)$ corresponding to $\sN \setminus \sN'$ via the Koenig-Yang correspondences \cite{KY} (cf. Theorem~\ref{thm:koenig-yang}; see also \cite[\S 4]{DDC2}). Then
\[
\sC_{\sH_\sN,J} \subseteq \overline{\sC_{\sH_\sM,I}} \iff
\sR_{I}(\sH_\sM) \geq \sR_{J}(\sH_\sN) \geq \sH_\sN \geq \sH_\sM \iff
(\sM,\sM') \geq (\sN,\sN'),
\]
where $\sR_{I}(\sH_\sM)$ is the right HRS tilt of $\sH_\sM$ at the torsion pair whose torsion class is generated by the simple objects $\mcS(\sH_\sM)\setminus I$.
\end{lemma}

It is well known that if $X$ is a regular CW complex then there is a homeomorphism from the classifying space of the poset of strata, $BP(X)$, to $X$. In \cite{CW-strat}, the following generalisation is obtained for regular, totally normal CW cellular stratified spaces; see \cite[\S 2.2-2.3]{CW-strat} or \cite[\S 2.7]{QW} for the definition.

\begin{theorem}[{\cite[Thm. 2.50]{CW-strat}}] \label{thm:CW-strat}
If $X$ is a regular, totally normal, CW cellular stratified space, then there is a homotopy equivalence $X \simeq BP(X)$.
\end{theorem}

\begin{proof}[Proof of Corollary~\ref{cor:contractible}]
If $\Lambda$ is silting-discrete then every bounded t-structure on $\Db(\Lambda)$ is algebraic, whence every stability condition $\sigma = (Z,\sH)$ is algebraic. By \cite[Prop. 3.21]{QW}, the cellular stratification of $\Stab(\Db(\Lambda))$ defined above is a regular, totally normal, CW-cellular stratification. By Lemma~\ref{lem:poset} and Theorem~\ref{thm:CW-strat}, we have
\[
\Stab(\Db(\Lambda)) \simeq BP(\Stab(\Db(\Lambda))) \cong B\bP_2(\Kb(\proj(\Lambda))),
\]
which, by Theorem~\ref{thm:contractible}, is contractible.
\end{proof}


\end{document}